\documentclass[preprint,12pt]{elsarticle}

\usepackage{amsthm,amssymb,amsmath}
\usepackage{tikz}
\usetikzlibrary{shapes,matrix,arrows,calc,through}
\usetikzlibrary{positioning,fit,shapes.misc,matrix,decorations.text,shapes.geometric}

\newtheorem{theorem}{Theorem}
\newtheorem{lemma}[theorem]{Lemma}
\newtheorem{corollary}[theorem]{Corollary}
\newtheorem{remark}[theorem]{Remark}

\newtheorem{openquestion}[theorem]{Open Question}
\newtheorem{example}[theorem]{Example}

\journal{Discrete Applied Mathematics}

\begin{document}

\begin{frontmatter}



\title{The local metric dimension\\ of subgraph-amalgamation of graphs}


\author[1]{Gabriel A. Barrag\'an-Ram\'irez}
\author[2]{Rinovia Simanjuntak}
\author[2]{Suhadi W. Saputro}
\author[2]{Saladin Uttunggadewa}

\address[1]{Departament d Enginyeria Inform\`{a}tica i Matem\`{a}tiques\\
Universitat Rovira i Virgili, Av. Pa\"{i}sos Catalans 26, 43007 Tarragona, Spain\\
gabrielantonio.barragan@estudiants.urv.cat}
\address[2]{Combinatorial Mathematics Research Group\\
Institut Teknologi Bandung, Jl. Ganesha 10 Bandung 40132, Indonesia\\
rino@math.itb.ac.id, suhadi@math.itb.ac.id, s\_uttunggadewa@math.itb.ac.id}

\begin{abstract}
A vertex $v$ is said to distinguish two other vertices  $x$ and $y$ of a
nontrivial connected graph G if the distance from $v$ to $x$ is different from the distance
from $v$ to $y$. A set $S\subseteq V(G)$ is a local metric set for $G$ if every two adjacent
vertices of $G$ are distinguished by some vertex of $S$. A local metric set with
minimum cardinality is called a local metric basis for $G$ and its cardinality, the local
metric dimension of $G$, denoted by $\dim_l(G)$. In this paper we present tight bounds
for the local metric dimension of subgraph-amalgamation of graphs with special emphasis
in the case of subgraphs which are isometric embeddings.
\end{abstract}


\end{frontmatter}



\section{Introduction}

All graphs considered in this paper are finite, simple, and non-null graphs. If $G$ is a graph we denote by $V(G)$ and $E(G)$ its
vertex and edge sets respectively. If $u$ and $v$ are the end vertices of an edge we write $uv$ for the edge itself.
If $uv\in E(G)$ we say that $u$ is \emph{adjacent} to $v$ (in $G$) and write $u\sim_G v$, omitting the
sub index when there is no ambiguity.


Given two graphs $G$ and $H$ such that $V(G)\cap V(H)=\emptyset$ we define the \emph{join} of $G$ and $H$ as the
graph $G+H$ with $V(G+H)=V(G)\cup V(H)$ and $E(G+H)=E(G)\cup E(H)\cup\{uv: u\in V(G), v\in V(H)\}$. Now we can define the
\emph{generalized fan} on $m+n$ vertices as $F_{m,n}:= \overline{K_m}+P_n$ where
$K_m$ is the complete graph on $m$ vertices and $P_n$ is the path on $n$ vertices.
The \emph{generalized wheel} on $m+n$ vertices is $W_{m,n}:= \overline{K_m}+C_n$ where $C_n$ is the \emph{cycle} on $n$
vertices.

With respect to local metric dimension of a graph we define the following terms. For $u,v\in V(G)$ we define the \emph{distance between }$u$ and $v$, denoted by $d(u,v)$, as the length of the shortest $u-v$ path. $A\subseteq E(G)$ is \emph{distinguished} by $B\subseteq V(G)$, if for each $uv\in A$ there exists $w\in B$ such that $d(w,u)\ne d(w,v)$. If $B\subseteq V(G)$ distinguishes $E(G)$ we say that $B$ is a \emph{local metric set}  for $G$. A local metric set with minimum cardinality is called a \emph{local metric basis} and its cardinality is called the \emph{local metric dimension}, denoted by $\dim_l(G)$.

The study of local metric dimension of graphs was introduced in \cite{Okamoto}. Among its numerous results we cite.

\begin{theorem}\label{okamoto bip}\cite{Okamoto}
$\dim_l(G)=1$ if and only if $G$ is a connected bipartite graph
\end{theorem}

\begin{theorem}\label{ocamoto complete}\cite{Okamoto}
$\dim_l(G)\le |V(G)|-1$ with equality if and only if $G$ is a complete graph
\end{theorem}

Since determining local metric dimension of a general graph $G$
is an NP-complete problem \cite{Fernau}, several researchers study the relation between local metric dimension of graphs
resulting from some product graphs with the local metric dimension of the operand graphs (\cite{elementar},\cite{corona},
 \cite{corona2}).

In this paper we are interested in subgraph-amalgamation that we proceed to define. A graph $J$ is an \emph{induced subgraph} of a graph $G$ if there exists an injective function $\iota: V(J)\rightarrow V(G)$, such that $\iota(u)\iota(v)\in E(G)$ if and only if $uv\in E(J)$. If $J$ is a common induced subgraph for the family of graphs $\{G_i\}$, we denote by $J_i$ the subgraph of $G_i$ induced in $G$, or in other words the image of the embedding function $\iota_i :V(J)\rightarrow V(G_i)$. If $V(J)=\{u_j\}$, we write $V(J_i)=\{u^i_j\}$, where $u^i_j=\iota_i(u_j)$.

Let $\{G_i\}$ be a family of graphs with common induced subgraph $J$ and embeddings $J_i$, where $V(G_i)=\{u^i_j\}$, $V(J)=\{v_k\}$,  and $V(J_i)=\{v^i_k\}$. We define the \emph{subgraph-amalgamation of $\{G_i\}$ over $J$} or \emph{$J$-amalgamation of $\{G_i\}$}, denoted by $\amalg\{(G_i|J_i)\}$, as the graph with vertex-set $$V(\amalg\{(G_i|J_i)\})= \bigcup V(G_i-J_i) \bigcup V(J)$$ and edge-set $$E(\amalg\{(G_i|J_i)\}) = \bigcup E(G_i-J_i) \bigcup E(J) \bigcup \{u^i_jv_k| v_k\in V(J), u^i_jv^i_k \in E(G_i)\}.$$ This means we identify the vertices in the corresponding copies of $J_i$ and preserve the adjacencies (here the importance of $J$ being an \emph{induced} subgraph).

In the case that the common induced subgraph is a vertex or an edge we speak about \emph{vertex-amalgamation} or \emph{edge-amalgamation} respectively. Previous work related with vertex-amalgamation could be found in \cite{elementar}, from where we have the following results.

\begin{theorem}\label{J=1}\cite{elementar}
Let $\{G_i\}$ be a family of connected graphs and $v_i\in V(G_i)$. If $H:=\amalg\{(G_i|v_i)\}$ then\\
\[\dim_l(H) = \left\{
  \begin{array}{ll}
    1, & \hbox{if all $G_i$s are bipartite;} \\
    \dim_l(G_1), & \hbox{if exactly one of the $G_i$s, say $G_1$, is non-bipartite;} \\
    \sum (\dim_l(G_i)-\epsilon_i), & \hbox{if at least two of the $G_i$s are non-bipartite,}\\
&\hbox{where $\epsilon_i= 1$ if $v_i$ is in a local metric basis of $G_i$}\\
&\hbox{and $0$ otherwise.}
  \end{array}
\right.\]
\end{theorem}

We could also see the subgraph-amalgamation of graphs as a pushout in the category of graphs, which is a colimit of a diagram consists of $n$ morphisms $\iota_i : J\rightarrow G_i$ with common domain $J$. For $n=2$, it means that there exist morphisms $f_1$ and $f_2$ for which the following diagram commutes.

\begin{center}
\begin{tikzpicture}[description/.style={fill=white,inner sep=2pt}]
\matrix (m) [matrix of math nodes, row sep=3em,
column sep=2.5em, text height=1.5ex, text depth=0.25ex]
{\amalg\{(G_i|J_i)\} & & G_2 \\
G_1& & J\\ };

\path[<-,font=\scriptsize]
(m-1-1) edge node[auto] {$f_2 $} (m-1-3);
\path[<-,font=\scriptsize] (m-1-1) edge node[auto] {$ f_1 $} (m-2-1);
\path[<-,font=\scriptsize] (m-1-3) edge node[auto] {$ \iota_2 $} (m-2-3);
\path[<-,font=\scriptsize] (m-2-1) edge node[auto] {$ \iota_1 $} (m-2-3);
\end{tikzpicture}
\end{center}

Moreover the subgraph-amalgamation of the family of graphs is universal in respect to the afore-mentioned diagram. Amalgamation occurs in different areas of mathematics, for instance in model theory where it plays a fundamental role in Fra\"{i}ss\'e's theorem \cite{Fraisse} which characterizes classes of countable homogeneous structures. Amalgamation is also one of the natural operations in data-bases \cite{Boyanczik}. Other applications of amalgamation could be found in \cite{ehrig} and \cite{nese}.


\section{Subgraph-amalgamation of Bipartite Graphs}

Obviously $$\dim_l(\amalg\{(G_i|J_i)\})\ge 1$$ and we reach that bound when $\amalg\{(G_i|J_i)\}$ is a connected bipartite graph (Theorem \ref{okamoto bip}). A necessary condition for $\amalg\{(G_i|J_i)\}$ to be bipartite is that all $G_i$s are also bipartite.

\begin{remark}\label{dim 1}
Let $\{G_i\}$ be a family of connected bipartite graphs and $J\cong K_1$ or $K_2$. If $H:=\amalg\{(G_i|J_i)\}$ then $\dim_l(H) =1$.
\end{remark}
\begin{proof}
The result for $J\cong K_1$ follows from Theorem \ref{J=1}. For $J\cong K_2$, let $\{U_i,V_i\}$ be the partition classes of $G_i$ and $V(J_i)=(u^i,v^i)$, where $u^i \in U_i$ and $v^i \in V_i$. Then $H$ is bipartite with partition classes $\{\bigcup U_i, \bigcup V_i\}$.
\end{proof}

In Corollary \ref{bipartitas mas bipartita}, we shall provide generalization of this remark. However, we could also amalgamate two bipartite graphs to obtain a non-bipartite graph, as in the following example.
\begin{example}\label{1+1=2}
Let $G_1\cong P_3$ and $G_2\cong P_4$, where $V(G_1)=\{u_i\}$, and $V(G_2)=\{v_i\}$.  $J\cong\overline{K_2}$ is a common induced subgraph for $\{G_1,G_2\}$, where $V(J_1)=(u_1,u_3)$ and $V(J_2)=(v_1,v_4)$ then $\amalg\{(G_i|J_i)\}\cong C_5$ and so $\dim_l(\amalg\{(G_i|J_i)\}) = 2$.
\end{example}

In fact local metric dimension of subgraph-amalgamation of two bipartite graphs could be as great as desired as we see in the next example. Recall that a \emph{spider} is a tree with a single vertex of degree at least three (called the \emph{head of the spider}) and remaining vertices of degree at most two. If we consider the multiset $\{m_i^{\alpha_i}\}$, we denote by $Sp{\{m_i^{\alpha_i}\}}$ the (unique) spider with exactly $\alpha_i$ pendant vertices at distance $m_i$ from the head.

\begin{example}\label{amal bipartites grande}
Let $G_1\cong Sp{\{2^n\}}$ and $G_2\cong Sp{\{3^n\}}$. Define $V(G_1)=\{u^i_j\}$ and $V(G_2)=\{v^i_j\}$, where the sub-index indicates the distance of the vertex to the head of the spider. If $J\cong \overline{K_{n+1}}$, with $V(J_1)=(u_0^1,\{u^i_2\})$ and $V(J_2)=(v_0^1,\{v^i_3\})$ then $H:=\amalg\{(G_i|J_i)\}\cong\amalg\{(H_i|u_0^1)\}$, where $\{H_i\}$ is a family of $n$ cycles on 5 vertices. Then by Theorem \ref{J=1}, $\dim_l(H) = n$.
\end{example}

On the other hand, if any of the $G_i$'s is not bipartite, neither is $\amalg\{(G_i|J_i)\}$ and, in this case,
$\dim_l(\amalg\{(G_i|J_i)\})\ge 2$.

\section{General Lower and Upper Bounds}

Let $J$ be an induced subgraph of a graph $G$. We define and edge $uv$ in $E(G-J)$ as \emph{parallel to $J$} in $G$, if for every $w\in V(J)$, $d(w,u)=d(w,v)$. We denote the set of parallel edges to $J$ in $G$ as $\parallel(J:G)$.
A set $S\subseteq V(G)$ is \emph{out-solving for $J$} if for every edge $uv\in E(J)$, there exists $s\in S$ such that
$d(s,u) \ne d(s,v)$. An out-solving set is said to be \emph{minimal} if it has the smallest cardinality.

From now on, let us consider $H=\amalg\{(G_i|J_i)\}$.

\begin{lemma}\label{adonis}
If for a vertex $a\in V(H)$ and an edge $uv\in \parallel(J_i:G_i)$, $d_H(a,u) \ne d_H(a,v)$, then
$a\in V(G_i-J_i)$.
\end{lemma}
\begin{proof}
Let $a\in V(H)$ and $uv\in \parallel(J_i:G_i)$ such that, without lost of generality, $d_H(a,u) = d_H(a,v)-1$.
For a contradiction, assume that $a\notin V(G_i-J_i)$. Let $P$ be a minimal path between $a$ and $u$ and consider $u_1\in
V(T) \cap V(J)$. Thus $d_H(a,v) \le d_H(a,u_1)+ d_H(u_1,v) = d_H(a,u_1)+ d_H(u_1,u) = d_H(a,u)$, a contradiction.
\end{proof}

Lemma \ref{adonis} states that for each $G_i$ we have a set $T_i \subseteq V(G_i-J_i)$ such that for every edge $uv\in \parallel(J_i:G_i)$, there exists a vertex $a\in T_i$ such that $d_H(a,u) \ne d_H(a,v)$. Such a set $T_i$ with minimum cardinality is called a \emph{traversal for $\parallel(J_i:G_i)$} or, simply, a traversal when there is no ambiguity.

\begin{remark}\label{suma de traversales}
A set $T$ distinguishes $\cup \parallel(J_i:G_i)$ if and only if for each $G_i$, $T \cap V(G_i)$
distinguishes $\parallel(J_i:G_i)$.
\end{remark}

The afore-mentioned remark leads to a lower bound for $\dim_l(H)$.
\begin{theorem} \label{inferior}
If $T_i$ is a traversal for $\parallel(J_i:G_i)$ and $S$ is a minimal out-solving set for $J$ then
$$\dim_l(H) \ge \sum |T_i| + |S|$$
and the bound is tight.
\end{theorem}
\begin{proof}
Assume for the contrary that there exists a local metric basis $C$ for $H$ such that $|C| < |\cup T_i \cup S|$. Then $C$ must be an out-solving set for $J$. For each $G_i$ we define $V(C_i) = V(C) \cap V(G_i)$ and so there exists a $C_k$ such that $|C_k| < |T_k|$. By Remark \ref{suma de traversales}, $C_k$ is not a traversal for $\parallel(J_k:G_k)$. Thus there exists $uv\in \parallel(J_k:G_k)$ which is not resolved by $C_k$. From Lemma \ref{adonis}, there is no vertex in $V(C) - \bigcup_{i \ne k} V(G_i-J_i)$ resolving such an edge. That is a contradiction and the result follows.

The tightness of the bound could be seen in the following example. Let $G_i \cong F_i + J$ where $\{F_i\}$ is a collection of at least two graphs and $J$ is an arbitrary graph. Then for $u_0 \in J$ and $v_0 \in F_i$ we have that $\dim_l(H)= \sum (\dim_l(u_0+F_i)-\epsilon_i) + (\dim_l(v_0+J)-\epsilon_J)$, where $\epsilon_i=0$ if $u_0$ is not contained in any basis of $u_0 + F_i$ or $1$ otherwise, and $\epsilon_J$ is defined in an analogous way. On the other hand, since $\parallel(J_i|G_i)= E(F_i)$, then we could choose $B_i-\{u_0\}$ as a traversal for $\parallel(J_i|G_i)$, where $B_i$ is a local metric basis for $u_0+F_i$. We also have that, for $S$ an out-solving set for $J$, $S\cap V(M_i)=\emptyset$ and so we could choose $B_J-\{v_0\}$ as a minimal out-solving set for $J$, where $B_J$ is a local metric basis for $v_0+J$. This leads to $\dim_l(H) = \sum |T_i| + |S|$.
\end{proof}

Another example for the tightness of the bound in Theorem \ref{inferior} is as follow.
\begin{example} \label{prismas} Let $G_1\cong W_{1,n}$ be a wheel on $n+1$ vertices, $G_2\cong C_{n}\square K_2$ be a prism on $2n$ vertices, and  $J \cong C_{n}$. Then $\dim_l(H)= \left\lceil\frac{n}{4}\right\rceil = 0 + 0 + \left\lceil\frac{n}{4}\right\rceil = |T_1|+ |T_2|+ |S|$.
\end{example}

By Example \ref{amal bipartites grande} we could see that we could not state a general upper bound for the local metric dimension of a subgraph-amalgamation of graphs as a function of the local metric dimension of the amalgamated graphs. In that example, the order of the induced subgraph $J$ is unbounded, however the next example shows that even if both the local metric dimension of the amalgamated graphs and the order of $J$ are bounded, the local metric dimension of the subgraph-amalgamation graph remains unbounded.

\begin{example}\label{watermelon}
We start with constructing the first amalgamated graph. Consider a cycle of order 17, with vertex-set $\{u_i\}$. We introduce  a new vertex $v$, together with four edges $u_1u_9$, $u_5u_{14}$, $vu_5$, and $vu_{14}$. We shall denote the resulting graph with $Q$. For $n\ge 4$, let $\{Q_i\}$ be a family of $n$ graphs with $Q_i \cong Q$. Let $I\cong K_2\cup K_1$ where $V(I)=\{a,b,c\}$ and $E(I)=\{ab\}$. Define $G_1=\amalg\{(Q_i|I_i)\}$ with $V(I_i) = \{u^i_1,u^i_9,v^i\}$ and $E(I_i)=\{u^i_1u^i_9\}$. We have $\dim_l(G_1)=2$ since the vertices $a$ and $b$ distinguish $E(G_1)$. Let the second amalgamated graph be $G_2 \cong v_0 + P_3$ with $V(G_2)=\{v_0,w_1,w_2,w_3\}$. Clearly, $\dim_l(G_2)=2$ being $\{v_0,w_2\}$ is a local metric basis. Now we are ready to define the subgraph-amalgamation $H:=\amalg\{(G_i|J_i)\}$, where $J\cong \overline{K_2}$, $V(J_1)=(b,c)$, and $V(J_2)=(w_1,w_3)$.

To show that the local metric dimension of $H$ is unbounded, consider the edge $u^i_5u^i_{14}$ for an arbitrary fixed $i$. We have $d_H(a,u^i_5) = 4 = d_H(a,u^i_{14})$, $d_H(b,u^i_5) = 3 = d_H(b,u^i_{14})$, $d_H(c,u^i_5) = 1 = d_H(c,u^i_{14})$, and $d_H(z,u^i_5)= 2 = d_H(z,u^i_{14})$, if $z\in \{v_0,w_2\}$. For any vertex in $z\in V(Q_j)$, $j\ne i$, the following hold. $d_H(z,u^i_5) =  5 = d_H(z,u^i_{14})$, if $z \in N(a) - \{b\}$, or $d_H(z,u^i_5) =  k = d_H(z,u^i_{14})$, where $k \in \{3,4\}$, if $z\in N(b)$, or $d_H(z,u^i_5) =  \min\{d_H(z,u^j_5),d_H(z,u^j_{14})\} + 2 = d_H(z,u^i_{14})$, for other $z$. We conclude that to resolve $u^i_5u^i_{14}$ we need a vertex in $Q_i-J$. Since $v_0w_2\in \parallel(S_2:G_2)$, $\dim_l(H)= n+1$.
\end{example}

In general we have a very crude upper bound.
\begin{theorem}\label{cruda}
If $\{G_i\}$ is a family of $n$ graphs and $n_H=|V(H)|$ then $$\dim_l(H) \le n_H - (n+1)$$ and the bound is tight.
\end{theorem}
\begin{proof}
Let $u\in V(J)$ and $v_i$ be a chosen vertex in each $V(G_i-J_i)$. We define $A=\{v_i\}\cup\{u\}$ and $B=V(H)-A$. Clearly $|B|= |V(H)| - |A|= n_H - (n+1)$. Let $uv$ be an edge with $u,v\in V(H)-B$ where $u\in V(G_i-J_i)$ and $v\in V(J_i)$. For any vertex $w\in V(G_j-J_j)$, $j \ne i$, we have $d_H(w,u)=d_H(w,v)+1$. This confirms that $B$ is a local metric metric set for $H$.

For the tightness, consider $H:=\amalg\{(K_{m_i}|K_r)\}$, for $m_i$ and $r$ are positive integers with $m_i>r$. Then
$\dim_l(H) = \sum m_i - (r+1) (n-1) - 2 = n_H - (n+1)$.
\end{proof}

However we could amalgamate arbitrarily large bipartite graphs obtaining yet another bipartite graph and thus show that the
upper bound is, in a lot of cases, still far away from the real value.

\begin{example} \label{amal bipartites pequenya}
Let $\{m_i\}$ be a set of odd numbers greater than 1 and $G_i\cong P_{m_i}$. We define $J\cong \bar{K_2}$, with $V(J_i)=\{u^i_1,u^i_{m_i}\}$ where $u^i_1$ and $v^i_{m_i}$ are the pendants in the paths. We have $n_H - (n+1) = 2 + \sum (m_i-2) - (n+1) = 1 + \sum(m_i-3)$, but $\dim_l(H) = 1$, since $H$ is connected and bipartite.
\end{example}

In the next section, we shall consider a special case of subgraph-amalgamation where we could have more information about the local metric dimension.

\section{Subgraph-amalgamation with Isometric Embedding}

As in the previous section, let $\{G_i\}$ be a family of graphs with common induced subgraph $J$ and embeddings $J_i$, and we denote $H:=\amalg\{(G_i|J_i)\}$. If $\iota_i : J \rightarrow G_i$ is an embedding of $J$ in $G_i$ we call $J_i:=\iota_i(J)$ the image of the embedding and, for each vertex $v\in V(J)$, $v^i:=\iota_i(v)$. If $J$ is a subgraph of $G$, we say that $J$ is \emph{isometrically embedded} in $G$, if for every $u,v\in V(G)$, $d_J(u,v)=d_G(u,v)$. We remark that an isometrically embedded subgraph is also an induced subgraph. We also say that $\{(G_i|J_i)\}$ is \emph{isometric} if for every $G_i,G_j$ and $a,b\in V(J)$, $d_{G_i}(a^i,b^i)=d_{G_j}(a^j,b^j)$. In the next lemma, we shall show that the notions of isometrically embedded and isometric are equivalent.

\begin{lemma} \label{isometricity}
Let $G_i$ be a connected graph. $G_i$ is isometrically embedded in $H$ if and only if
$\{(G_i|J_i)\}$ is isometric.
\end{lemma}
\begin{proof}
The necessity of the condition is clear. Now suppose that for every $i,j$ and $a,b\in V(J)$, $d_{G_i}(a^i,b^i)= d_{G_j}(a^j,b^j)$. Consider $u,v\in V(G_i)$ and let $P_{uv}$ be a minimal path in $H$ between $u$ and $v$. If $P_{uv} \subseteq G_i$ we are done, otherwise there exist $j \ne i$ and $b \in V(G_j)$ such that $b \in P_{uv}$. Let  $u_1 \in V(J) \cap P_{ub}$  and $v_1 \in V(J) \cap P_{vb}$ such that $d_H(u_1,u)$ and $d_H(v_1,v)$ are minimal. By hypotheses there exist $P_{u_1v_1} \subseteq G_i$ a minimal path between $u_1$ and $v_1$ with $length(P_{u_1v_1})= d_{G_i}(u_1,v_1)= d_{G_j}(u_1,v_1) \le d_{G_j}(u_1,b) + d_{G_j}(b,v_1)$. If we define $P^*=P_{uu_1} \cup P_{u_1v_1} \cup P_{v_1v}$, then  $P^* \subseteq G_i$ and $P^*$ is a $u,v$ path with length at most $d_{G_j}(u,v)$, so $d_{G_i}(u,v) \le d_{G_j}(u,v)= length(P_{uv})$ and, as $P_{uv}$ is a minimal path, the result follows.
\end{proof}

The following gives a sufficient condition for isometric embedding.
\begin{lemma}\label{diam2 imp iso}
If diameter of $J$ is at most $2$ then every $G_i$ is embedded isometrically in $H$.
\end{lemma}
\begin{proof}
Assume that $G_i$ is not embedded isometrically in $H$, that means that there exist $u,v\in V(G_i)$ such that $d_{G_i}(u,v) \ne d_H(u,v)$, necessarily $d_H(u,v) < d_{G_i}(u,v)$. For Lemma \ref{isometricity} we could suppose that $u,v \in V(J_i)$, and so $d_H(u,v) < d_{G_i}(u,v) \le d_{J_i}(u,v) \le 2$. Thus $u\sim_H v$ and as $G_i$ is an induced subgraph of $H$, $u\sim_{G_1} v$, a contradiction.
\end{proof}

The next lemma generalizes result in Remark \ref{dim 1} on subgraph-amalgamation of connected bipartite graphs.

\begin{corollary}\label{bipartitas mas bipartita}
Let $\{(G_i|J_i)\}$ be isometric. If for every $i$, $G_i$ is connected and bipartite,  then $\dim_l(H)=1$.
\end{corollary}
\begin{proof}
For every $v \in V(J)$,  $\{v_i\}$  is a local metric set for $G_i$, then $\{v\}$ is a metric set for $H$. Thus $H$ is connected and bipartite, and therefore $\dim_l(H)=1$.
\end{proof}

\begin{corollary}\label{las bipartitas no cuentan}
Let ${\cal G}=\{G_i\}$ be a family of connected graphs with at least one $G_i$ non-bipartite and ${\cal H}\subseteq{\cal G}$ be the subset of all non-bipartite graphs in ${\cal G}$. If $J$ is isometrically embedded then $$\dim_l(\amalg_{\cal G}\{(G_i|J_i)\})= \dim_l(\amalg_{\cal H}\{(G_i|J_i)\}).$$
\end{corollary}

\begin{example}\label{cycles} By the afore-mentioned corollary, if $\{(C_i|J_i)\}$ is isometric, where not all $C_i$'s are of even order, then $\dim_l(\amalg\{(C_i|J_i)\})= 2$.
\end{example}

The importance of isometric embedding in determining local metric dimension of subgraph-amalgamation is illustrated in the following lemma.
\begin{lemma}\label{union de generadores}
Let $\{(G_i|J_i)\}$ be isometric. If $A_i\subseteq V(G_i)$ is a local metric set for $G_i$ then $\bigcup A_i$ is a local metric set for $H$.
 \end{lemma}
\begin{proof} Let $uv$ be an edge in $H$ and $u,v \in V(G_i)$. As $A_i$ is a local metric set for $G_i$, there exists $b\in A_i$
such that $d_H(b,u) = d_{G_i}(b,u) \ne d_{G_i}(b,v) = d_H(b,v)$. And the result follows.
\end{proof}

Example \ref{1+1=2} shows that isometricity is sufficient for the union of local metric sets of the amalgamated graphs to be a local metric set for the subgaph-amalgamation. Now we are ready to present upper bounds for subgraph-amalgamation with isometric embedings.

\begin{theorem}\label{cota sup}
If $\{(G_i|J_i)\}$ is isometric and $T_i$ is a traversal for $\parallel(J_i:G_i)$ then $$\dim_l(H) \le \min\{\sum \dim_l(G_i), \sum |T_i| + |J|\}$$ and the bound is tight.
\end{theorem}
\begin{proof}
$\dim_l(H) \le \sum \dim_l(G_i)$ is a direct consequence of Lemma \ref{union de generadores}. It is clear that $T_i\cup V(J)$ is a metric set for $G_i$. Now define $B=(\cup T_i) \cup V(J)$. By Lemma \ref{union de generadores}, $B$ is a local metric set for $H$ and the result follows.

We could see that the first part of the bound is tight with the following example. Let $G_1 \cong G_2 \cong \amalg\{(v_0+P_6|v_0),(C_4|u_1)\}$ . Consider the graph $H:=\amalg\{(G_i|C_4)\}$, then $\dim_l(H)= 4 = 2 + 2 = \dim_l(G_1)+ \dim_l(G_2)$.

The tightness of the second part of the bound is shown in the next example. Let $G_1\cong K_n$, $G_2\cong C_{2m+1}$, and $J\cong K_2$. Then $|T_1|= m-3$, $|T_2|=0$, $|J|=2$, and $\dim_l(H) = m-1 = \sum |T_i| + |J| < \dim_l(G_1) + \dim_l(G_2) = m+1$.
\end{proof}

We shall proceed with examining several conditions that ensure that the first bound in Theorem \ref{cota sup} is achieved. The first condition examines the relation between local metric basis of the amalgamated graphs.

\begin{lemma}\label{necesaria para suma de dimensiones}
Let $\{(G_i|J_i)\}$ be isometric. If $\dim_l(H)= \sum \dim_l(G_i)$ and $B_k$ is a local metric basis for $G_k$ then, for every $i \ne j$, $B_i \cap B_j = \emptyset$.
\end{lemma}
\begin{proof}
Suppose that there exist local metric basis $B_1$ and $B_2$ of $G_1$ and $G_2$ such that $B_1 \cap B_2 \ne \emptyset$. By Lemma \ref{union de generadores}, $\dim_l(H) \le |\cup B_i| = |(B_1 \cup B_2) \cup_{i>2} B_i| < \sum |B_i| = \sum \dim_l(G_i)$.
\end{proof}

However this lemma does not characterize the graphs whose subgraph-amalgamation's local dimension equals the sum
of their dimensions, as we could see in the following example.
\begin{example}
Let $G$ be a graph obtained from $K_4$, where $V(K_4)=\{u_i\}$, with two new vertices $v$ and $w$ with neighborhoods $N(v)=\{u_1,u_2\}$ and $N(w)=\{u_2,u_3\}$. Consider $G_i \cong G$ and $J\cong K_4$, which is isometrically embedded in $H$. The only local basis of $G_i$ is $B_i=\{v^i,w^i\}$, and clearly, for $i \ne j$, $B_i \cap B_j = \emptyset$. However, $\dim_l(\amalg\{(G_i|J_i)\})=2$.
\end{example}

The second condition deals with a relation between a local metric basis of an amalgamated graphs with an isometric embedding $J$.

\begin{lemma}\label{suficiente para suma de dimensiones}
Let $\{(G_i|J_i)\}$ be isometric. If for every local metric basis $B_i$ of $G_i$, $B_i \cap V(J_i) = \emptyset$ then $\dim_l(H) = \sum \dim_l(G_i)$.
\end{lemma}
\begin{proof}
Assume, for a contradiction, that $dim_l(H) < \sum \dim_l(G_i)$. Let $B$ be a local metric basis for $H$ and $B_i=B \cap V(G_i-J_i)$. Thus it is necessary for one of the $B_i$s, say $B_1$, $|B_1| < \dim_l(G_1)$ holds. So there exist $uv \in E(G_1)$ not distinguished by $B_1$. As $B$ is a local metric basis, there exist $b\in B - B_1$ such that, without lost of generality, $d(b,u) = d(b,v)-1$. Let $P_{bu}$ a minimal path between $b$ and $u$ and $b_1 \in P_{bu} \cap V(J)$. $b_1$ distinguishes $uv$ because otherwise $d(b,v) \le d(b,u_1) + d(u_1,v) = d(b,u) + d(u_1,u) = d(b,u)$, a contradiction. If $B_1 \cup \{b_1\}$ resolves $E(G_1)$ we obtain a contradiction because there exists a local metric basis for $G_1$ that intersects $V(J)$; otherwise we could repeat the procedure until a contradiction is obtained, and the result follows.
\end{proof}

The sufficient condition in Lemma \ref{suficiente para suma de dimensiones} is not necessary as shown in the next example.

\begin{example} Let $G_1 \cong G_2 \cong F_{1,9} \cong v + P_9$, $J \cong F_{1,3}$, where $V(G_i)=\{v,u^i_j\}$, $V(J_1)=(v^1_0,u^1_7,u^1_8,u^1_9)$, and $V(J_2)=(v^2_0,u^2_1,u^2_2,u^2_3)$. Then $H \cong F_{1,15}$ and $\dim_l(H) = 4 = \sum \dim_l(G_i)$. However each $G_i$ has only one local metric basis, namely, $B_i=\{u^i_3,u^i_7\}$, and so $B_1 \cap J = \{u^1_7\}$ and $B_2 \cap J = \{u^2_3\}$.
\end{example}

The last condition based on connection between local metric basis and traversal sets.

\begin{lemma}\label{Bi contained in Ti} If every traversal $T_i$ for $\parallel(J_i:G_i)$ contains a local metric basis for $G_i$, then $\dim_l(H) = \sum \dim_l(G_i)$.
\end{lemma}
\begin{proof}
Recall that every local metric set contains a traversal (Theorem \ref{inferior}) and the result follows.
\end{proof}

We could construct an infinite family of graphs that fulfill the conditions in Lemma \ref{Bi contained in Ti}, as we show in the
following. We recall that $\chi(G)$ denotes the \emph{chromatic number} of $G$.
\begin{lemma}\label{ex Bi contained in Ti}
Let $J$ be an arbitrary graph and $m \ge \chi(J)$.
Then there exists a graph $G$, such that
\begin{enumerate}
\item $J$ is an induced subgraph of $G$,
\item $\dim_l(G) = m $, and
\item  There exists $B$ a local metric basis for $G$ such that $B\cap V(J)=\emptyset$.
\end{enumerate}
\end{lemma}
\begin{proof}
We shall consider two separate cases.

\textbf{Case I:} $\chi(J)=1$. In this case $J\cong\overline{K_n}$. For $m=1$ take $G \cong J+K_1$ and $B=V(K_1)$.
For $m \ge 2$  we take $G \cong J + P_{4m+1}$. If $V(P_{4m+1})=\{u_i\}_{i=1}^{4m-1}$ then $B=\{u_{4i-1}\}_{i=1}^{m}$.

\textbf{Case II:} $\chi(J)\ge 2$. Let $\{U_i\}$ be the color classes of $J$ and $X:=\{x_iy_i\}$ be a set of $m$ new edges. Construct $G$ as the graph resulting from connecting $x_i$ and $y_i$ to a vertex $u\in U_j$ with edges $x_iu$ and $y_iu$ if $i \equiv j\; mod \;\chi(J)$ or with disjoint paths $P_{x_iu}$ and $P_{y_iu}$ of length $2$ otherwise. Then $J$ is an induced sugraph of $G$, $\parallel(J:G) = X$, and $T=\{x_i\}$ is a traversal for $\parallel(J:G)$. Clearly, $|T|=|X|=m$. We shall show that $T$ is also a local metric basis for $G$. First we observe that $T$ distinguishes $X$. Since edges in $X$ are disjoint and no vertex in $V(J)$ distinguishes an edge in $X$, then there is no set $S$ with $|S|<|T|$ could distinguish edges in $X$. Now consider $uv\in E(G) - X$. If $u=x_i$ then $x_i\in T$  distinguishes $u$ from $v$. If $u=y_i$ and $v\in U_j$ with $j\equiv i\; mod\; \chi(J)$, then for $x_k\in T$, $k\not\equiv j\; mod\; \chi(J)$, we have $d(x_k,v)= 2 \ne 3 = d(x_k,u)$. If  $u=y_i$ and $v\in P_{y_iw}$, where $w\in U_j$, then for $x_k\in T$, $k\equiv j\; mod\; \chi(J)$, we have $d(x_k,v)= 2 \ne 3 = d(x_k,u)$. If $uv \in E(J)$, say without lost of generality, $u\in U_i$ and $v\in U_j$, consider $x_k\in T$, $k\equiv i\; mod\; \chi(J)$. And so we have $d(x_k,u) = 1 \ne 2 = d(x_k,v)$. So $T$ is a local metric basis and the result follows.
\end{proof}

\subsection{A Better Upper Bound Using Co-traversals}

In spite of the tightness of the bound in Theorem \ref{cota sup}, we could find examples whose local metric dimensions are still much smaller than the bounds, as we can see in the following.

\begin{example}\label{la suma de dimensiones o la J son mucho}
For $i=1,..,n$ and $m>5$, let $G_i \cong v_0+ P_m$ and $J\cong P_m$. $\dim_l(G_i)=\left\lceil\frac{m-1}{4}\right\rceil$, and so $\sum \dim_l(G_i)= n\left\lceil\frac{m-1}{4}\right\rceil$. On the other hand, $|T_i|=0$, $|J|= m$, and $\sum |T_i| + |J| = m$. However, $\dim_l(H)=\left\lceil\frac{m-1}{4}\right\rceil$. In fact, the local metric dimension is equal to the lower bound given by Theorem \ref{inferior}.
\end{example}

This example motivates the following definition. Let $\{(G_i|J_i)\}$ be isometric and $T_i$ be a traversal for $\parallel(J_i:G_i)$. We say that $C\subseteq V(J)$ is a \emph{co-traversal}, if for every $uv\in E(G_i)-E(J)$ there exists $a \in C \cup T_i$ that distinguishes $uv$ and $C$ is such a set of minimum cardinality. It is clear that $V(J)$ is both co-traversal and out-solving. Thus $C$ is a co-traversal and $S$ is a minimal out-solving set for $J$ then $|C \cup S| \le |J|$ and we have the following statement which is an improvement for the upper bound.

\begin{theorem}\label{c s t}
Let $\{(G_i,J_i)\}$ be isometric with $T_i$ a traversal for $\parallel(J_i:G_i)$, $C$ a co-traversal, and $S$ a minimal out-solving set for $J$. Then $$\dim_l(H) \le \sum |T_i| + |C \cup S|$$ and the bound is tight.
\end{theorem}
\begin{proof}
The proof is an application of the definitions and Theorem \ref{cota sup}. Example \ref{la suma de dimensiones o la J son mucho} gives the tighness of the bound.
\end{proof}

\begin{corollary}\label{out come cotraversal}
Let $\{(G_i,J_i)\}$ be isometric with $T_i$ a traversal for $\parallel(J_i:G_i)$, $C$ a co-traversal, and $S$ a minimal out-solving set for $J$. If $C \subseteq S$ then $$\dim_l(H) = \sum |T_i| + |S|.$$
\end{corollary}
\begin{proof}
The proof is a direct consequence of Theorems \ref{inferior} and \ref{cota sup}.
\end{proof}

We shall give two examples for Corollary \ref{out come cotraversal}.

\begin{example}
Suppose that $K_5$ is a complete graph with vertex-set $\{u_i\}$. Let $G_1 \cong G_2 \cong K_5 - \{u_1u_3, u_1u_4\}$, $J\cong K_2$, and $J_i=(u^i_3,u^i_4)$. Then $\parallel(J_i:G_i)=\{u^i_2u^i_5\}$ and $C = S = \{u_3\}$, and thus $\dim_l(H)= 3 = \sum |T_i| + |S|$.
\end{example}

\begin{example}\label{ex out come cotraversal}
Let $J$ be an arbitrary graph. For $m_i\in \mathbb{N}^+$, let $G_i\cong J + \overline{K_{m_i}}$. We observe that, for every $G_i$, $\parallel(J_i:G_i)=\emptyset$, and so by Theorem \ref{c s t}, if $C$ is a co-traversal and $S$ is an out-solving set for $J$, $\dim_l(H) \le |C \cup S|$.

Let $v_1 \in V(\overline{K_{m_1}})$ and $B$ be a local metric basis for $J + v_1$. We shall show that $B$ is also a local metric basis for $H$. As $B$ is a local metric basis for $J + v_1$, $B - \{v_1\}$ distinguishes $E(J)$. Consider $uv\in E(H-J)$, necessarily $u \in V(\overline{K_{m_i}})$ and $v\in V(J)$. We have two cases. \textbf{Case I: $v_1\in B$.} In this case, if $u=v_1$ $d(v_1,u) = 0 \ne 1 = d(v_1,v)$, otherwise, $d(v_1,u) = 2 \ne 1 = d (v_1,v)$. \textbf{Case II: $v_1\notin B$.} This means that there exists $b\in B$ such that $b \not\sim v$, and so $d(b,v) = 2 \ne 1 = d(b,v_1)$. Thus, $B$ is a local metric set for $H$. If there exists $A\subseteq V(H)$ such that $A$ is a local metric set and $|A|<|B|$ that means that $A$ is not a local metric basis for $J + v_1$. Therefore there exists $uv \in E(v_1+J) \subseteq E(H)$ such that $A$ does not distinguishes $uv$, a contradiction. This leads to $B$ being a local metric basis and $\dim_l(H) = \dim_l(J + K_1)$.

On the other hand, if $v_1\in B$, $\{v_1\}$ is a co-traversal and $B-\{v_1\}$ is a minimal out-solving set for $J$.
If $v_1\notin B$, then $B$ is both a co-traversal and a minimal out-solving set for $J$.
\end{example}

Inspired by Example \ref{ex out come cotraversal} we could observe the two following more general results.
\begin{lemma}
Let $J \cong \overline{K_m}$ and $\{(G_i,J_i)\}$ be isometric. If $L_i = G_i-J_i$ then $$\dim_l(H) \le \sum \dim_l(L_i + K_1) + 1.$$
\end{lemma}

\begin{lemma}\label{ex out come cotraversal II}
Let $J$ be an arbitrary graph, $G_1 \cong J + \overline{K_m}$, with $m \ge 1$, and for $i>1$, $G_i$ be a graph  containing an induced subgraph $J_i \cong J$. If $v\in V(G_i)-V(J_i)$ implies $N(v)\subseteq V(J_i)$ and $\{(G_i|J_i)\}$ is isometric then $$\dim_l(J + K_1) \le \dim_l(H) \le \dim_l(J + K_1) +\epsilon,$$
where $\epsilon=1$ if there exists a basis of $J +K_1$ not containing the only vertex in $K_1$,
and $\epsilon=0$ otherwise.
\end{lemma}
\begin{proof}
Observe that, for every $G_i$, $\parallel(J_i:G_i)=\emptyset$. Therefore, by Theorem \ref{c s t}, if $C$ is a co-traversal for $H$ and $S$ is a minimal out-solving set for $J$ then $\dim_l(H)\le |C \cup S|$. Let $v_1 \in V(\overline{K_{m}})$ and $B$ a local metric basis for $J + v_1$. We shall show that $B\cup\{v_1\}$ is a local metric basis for $H$. As $B$ is a local metric basis for $J + v_1$, $B-\{v_1\}$ distinguishes $E(J)$. Consider $uv\in E(H-J)$, necessarily $u\in V(J)$ and $v\not\in V(J)$. If $v=v_1$ then  $d(v_1,u) = 1 \ne 0 = d(v_1,v)$ and otherwise, $d(v_1,u) = 1 \ne 2 = d(v_1,v)$. Thus $B\cup\{v_1\}$ is a local metric set. Now assume that there exists $A \subseteq V(H)$ such that $A$ is a local metric set and $|A|<|B\cup\{v_1\}|$. We shall consider two cases: \textbf{Case I: $v_1 \in B$.} In this case $A$ is not a local metric basis for $J + v_1$ and so there exists $uv\in E(J + v_1) \subseteq E(H)$ such that $A$ does not distinguishes $uv$, a contradiction. \textbf{Case II: $v_1 \notin B$.} To avoid the former contradiction, we must have $|A| = |B| = \dim_l(J+K_1)$, and the result follows.
\end{proof}

\subsection{Yet A Better Upper Bound Using Covers}

Again we could find examples where the bound in Theorem \ref{c s t} is still far away from the local metric dimension of the subgraph-amalgamation.

\begin{example}\label{dim 2}
For $m>1$ and $r=4m+1$, consider $G \cong v_0 + P_r$, where $V(P_r)=\{u_j\}$. Let $\{G_i\}$ be a set of $n$ graphs isomorphic to $G$ and $J \cong \overline{K_m}$, with $V(J_i)=(u^i_{4k+3})$. Notice that $J_i$ is a local metric basis for $G_i$ and so $\dim_l(G_i)= m$. We also have that $\parallel(J_i:G_i)=\{v_0^iu^i_{4k+2}, v_0^iu^i_{4k+4}\}$, $T_i=\{v_0^i\}$ is a traversal, $V(J)$ is a co-traversal, and $S=\emptyset$. Theorem \ref{c s t} gives $\dim_l(H) = \sum |T_i| + |C \cup S|= n + m$, however $\dim_l(H)=n$, since $\{v_0^i\}$ is a local metric basis for $H$.
\end{example}

In a quest for a better upper bound, we define the following. For $uv\in E(G)$, we say that $uv$ is \emph{solvable by} $J$, if $d(u,J) \ne d(v,J)$ and, in this case, we write $uv \in  Solv(J:G)$.

\begin{lemma}\label{distinguish cover for projections}
Let $\{(G_i,J_i)\}$ be an isometric amalgam, $uv\in Solv (J_i:G_i)$, where  $d(u,J)<d(v,J)$, and $c\in V(J)$. If $d(u,c)=d(u,J)$ then $c$ distinguishes $uv$. Moreover, if for $j \ne i$ there exists $s\in V(G_j)$ such that $d(s,c)=d(s,J)$ then $s$ distinguishes $uv$.
\end{lemma}
\begin{proof} The first result follows form $d(v,c) \ge d(v,J) > d(u,J) = d(u,c)$. If $s\in V(J_j)$ then the second result is trivial. Now let $s\in V(G_j-J_j)$ and suppose, for a contradiction, $d(s,u)=d(s,v)$. Let $P_{su}$ and $P_{sv}$ be the two minimal paths from $s$ to $u$ and from $s$ to $v$. Let $u_1 \in P_{su} \cap V(J)$ and $v_1 \in P_{sv} \cap V(J)$. By the definition of $c$, $d(s,u_1)+d(u_1,u) = d(s,u) \le d(s,c)+d(c,u) \le d(s,u_1) + d(u_1,u)$ and so $d(s,u)=d(s,c)+d(c,u)$. On the other hand, $$d(s,v_1)+d(v_1,v) = d(s,v) \le d(s,c) + d(c,v) \le d(s,v_1) + d(c,u) +1$$ $$ \le d(s,v_1) + d(v,J) \le d(s,v_1) + d(v_1,v).$$ This leads to $d(s,v)=d(s,c)+d(c,v)$. Since $d(s,v)=d(s,u)$, we have $d(c,u)=d(c,v)$, a contradiction, and the result follows.
\end{proof}

The results in Lemma \ref{distinguish cover for projections} lead us to the following definitions. A co-traversal $C$ is said to be $\emph{projective}$ if for every $uv\in Solv(J_i:G_i)$, where $d(u,J)<d(v,J)$ and $T_i$ does not distinguish $uv$, there exists a vertex $c\in C$ such that $d(u,c)=d(u,J)$. For a co-traversal $C$ we say that a set $\bar{C_i} \subseteq V(G_i)$ is a \emph{$C_i$-cover} if, for every $c\in C$, there exists a vertex $a \in \bar{C_i}$, such that $d(a,c)=d(a,J)$. We shall categorize the $C_i$-covers into two types as follow. A $C_i$-cover $\bar{C_i}$ is called \emph{complete} if  $\bar{C_i}$ distinguishes $E(G_i-J_i)-(\parallel(J_i:G_i)\cup Solv(J_i:G_i))$. A $C_i$-cover is called \emph{self-resolving} if $\bar{C_i}$ distinguishes $(E(G_i-J_i))-\parallel(J_i:G_i)$. For convenience, if $(E(G_i-J_i)) = (\parallel(J_i:G_i)\cup Solv(J_i:G_i))$, we consider $\emptyset$  a complete  $C_i$-cover.

If $\{(G_i|J_i)\}$ is isometric then for every $G_i$, a co-traversal $C$ is self-resolving. Moreover, if $(E(G_i-J_i)) \ne (\parallel(J_i:G_i) \cup Solv(J_i:G_i))$, $C$ is also complete. Therefore for any co-traversal $C$, we could always find a set ${\cal C}=\{\bar{C_i}\}$ of complete $C_i$-covers such that $|\cup C_i| \le |C|$ and the following theorem is an improvement to our previous upper bounds.

\begin{theorem}\label{complete covers}
Let $\{(G_i|J_i)\}$ be isometric, $C$ be a projective co-traversal, $T_i$'s be traversals, and $S$ be a minimal out-solving set for $J$. Suppose that ${\cal C}=\{\bar{C_i}:\bar{C_i}\subseteq V(G_i) \}$ is a set of complete $C_i$-covers for $C$. If either there exists a $\bar{C_i}\in {\cal C}$ such that $\bar{C_i}$ is self-resolving or for every $c\in C$ there exist two vertices $x\in \bar{C_i}$ and $y\in\bar{C_j}$, where $i \ne j$ such that $d(x,c)=d(x,J)$ and $d(y,c)=d(y,J)$ then $$\dim_l(H)\le \sum |T_i| + |S| + \sum |\bar{C_i}|$$ and the bound is tight.
\end{theorem}
\begin{proof}
For $uv \in E(H)$ we have the following four cases. If $uv \in E(J)$, then it will be distinguished by $S$. If $uv\in \parallel(J_i:G_i)$, then $T$ distinguishes $uv$. If $uv\in (E(G_i-J_i))-(\parallel(J_i:G_i) \cup Solv(J_i:G_i))$, it will be distinguished by  ${\cal C}$. The last case is if $uv\in Solv(J_i:G_i)$. We could suppose without lost of generality that $d(u,J)<d(v,J)$. If $\bar{C_i}$ is self-resolving then $\bar{C_i}$  distinguishes $uv$. Otherwise, as $C$ is projective, there exists $c\in C$ such that $d(c,u)=d(u,J)$ and $x \in \bar{C_j} \ne \bar{C_i}$ such that $d(x,c)=d(x,J)$. Then, by Lemma \ref{distinguish cover for projections} $x\in \bar{C_j}$ distinguishes $uv$.

Example \ref{dim 2} shows that the bound is tight.

We also have another example showing the tightness of the bound and this example uses a self-resolving cover. Let $V(G_1)=\{u_1, \ldots, u_5\}$, where $G_1 \cong K_5 -\{u_1u_3, u_1u_4\}$, $G_2\cong C_5$, where $V(G_2)=\{v_i\}$, and $J\cong K_2$, where $V(J_1)=(u_4,u_5)$ and $V(J_2)=(v_3,v_4)$. In this case, $\parallel(J_1:G_1)=\{u_2u_3\}$ and  $\parallel(J_2:G_2)=\{\}$. We could choose  $T_1=\{u_1\}$ with the advantage that it is both a traversal and a minimal out-solving set for $J$. As $Solv(J_1:G_1)=E(G_1)-\{u_1u_2, u_2u_3, u_4u_5\}$, $Solv(J_2:G_2)=E(G_2)-\{v_3v_4\}$, the co-traversal $C=V(J)=\{u_4,u_5\}$ is projective. We also have $\emptyset$ as a complete $C_1$-cover and $\{v_1\}$ as a self-resolving $C_2$-cover. Then $T_1 \cup S \cup \bar{C_i} = \{u_1,v_1\}$ is a local metric set for $H$ and, as $H$ is not bipartite, $B$ is a local metric basis for $H$.
\end{proof}


If the hypotheses of Theorem \ref{complete covers} are not fulfilled then we have the following result.
\begin{lemma} Let $\{(G_i|J_i)\}$ be isometric, $T_i$ be a traversal for $G_i$, $S$ be a minimal out-solving set for $J$, and $C$ be a projective co-traversal. Suppose that $\bar{C_i}$ be a cover such that $|\cup T_i \bigcup S \cup \bar{C_i}|$ is minimum. If for every $i$, $E(G_i-J_i) = \parallel(J_i:G_i)\cup Solv(J_i:G_i)$ and no $\bar{C_i}$ is self-resolving then $$\dim_l(H)=\sum |T_i| + |S| + \sum |\bar{C_i}|.$$
\end{lemma}

We could also extend an existing subgraph amalgamation to a new sub-graph amalgamation with a prescribed local metric dimension as stated in the following.
\begin{lemma} Let $\{(G_i|J_i)\}$ be isometric such that
 $E(G_i-J_i) = \parallel(J_i:G_i) \cup Solv(J_i:G_i)$.
Let $T_i$ be a traversal for $\{(G_i|J_i)\}$ and let $S$ be a minimal out-solving set for $J$. If there exist $C$ a projective co-traversal, then
for every positive integer $r$ there exists a graph $A$ containing $J$ as induced subgraph such that
$\dim_l(\amalg\{(G_i'|J_i')\}) = \sum |T_i| + |S| + r$, where $\{G_i'\}=\{G_i\}\cup\{A\}$.
\end{lemma}
\begin{proof}
Let $A'= J + K_r$. To ensure isometricity we substitute each edge $uv$, with $u \in V(K_r)$ and $v \in V(J)$ with a path $P_{uv}$ of length $D(J)$, where $D(J)$ denotes the diameter of $J$. The resulting graph is the $A$ we are looking for.$\parallel(J:A)=E(K_r)$. If we denote $V(K_r)=\{u_i\}$, then we can choose as a traversal $T_A=\{u_1, \ldots, u_{r-1}\}$, and so $\{u_r\}$ is a self-resolving cover and the result follows.
\end{proof}

In fact self-resolving covers are not infrequent as we could see in the next lemma.
\begin{lemma} \label{contruccion self resolving} Let $\{(G_i:J_i)\}$ be isometric. If for some $G_i$, the graph
$G_i-E(J_i)$ is bipartite, then there exists $A \subseteq V(G_i)$ such that $A$ is a self-resolving cover for $V(J)$.
\end{lemma}
\begin{proof}
Let ${\cal A}$ be the set of connected components of $G_i-E(J_i)$. For each $A_i \in {\cal A}$ let $M_i:=\{a_i \in A_i: b_i\in N(a_i) \Rightarrow d(b_i,J) \le d(a_i,J)\}$. For every $x\in V(J) \cap A_i$, we construct recursive sets $D_{x,k}$ with $D_{x,0}=\{x\}$ and $D_{x,k+1} = \{y \in A_i: d(y,J)=k+1 \wedge N(y)\cap D_{x,k}\ne \emptyset\}$. Eventually, for sufficiently large $k_0$, $D_{x,k_0} \cap M_i \ne \emptyset$. For $a_i\in D_{x,k_0} \cap M_i$ we have $d(a_i,J)=d(a_i,x)$ and so $M_i$ is a cover for $J$. Let $uv\in E(G_i-J_i)$, with $d(u,J) \le d(v,J)$. Then there exists an integer $r_0$ such that $M_i\cap D_{v,r_0} \ne \emptyset$. Since  $G_i-E(J_i)$ is bipartite, then for $a_i \in M_i\cap D_{v,r_0}$, $d(a_i,v)=r_0 - d(v,J) \ne r_0 - d(v,J) + 1 = d(a_i,u)$. This results in $M_i$ being self-resolving.
\end{proof}

We could utilize the former lemma to obtain a new upper bound. For an isometric $\{(G_i|J_i)\}$, define the class ${\cal B} = \{G_i: G_i-E(J_i)\; \hbox{is a bipartite graph}\}$. For $G_i\in {\cal B}$, let $M_i:=\{a_i\in V(G_i): b_i\in N(a_i)\Rightarrow d(b_i,J)\le d(a_i,J)\}$.

\begin{lemma} Let $\{(G_i|J_i)\}$ be isometric. If $E(G_i-J_i) =\parallel(J_i:G_i)\cup Solv(J_i:G_i)$ and
${\cal B} \ne \emptyset$ then
$$\dim_l(H) \leq \sum |T_i| + |S| + \sum |M_i|.$$
\end{lemma}

The results we presented in this paper still lead to several open questions. Two of those which are of importance are presented here.
\begin{openquestion} What are the conditions such that our upper bounds are the exact local metric dimension? Or, does any basis of a subgraph-amalgamation contain a cover?
\end{openquestion}

\begin{openquestion} Is isometricity the least we could ask to obtain more information regarding the local metric dimension? What if  we only consider the connectivity of $J$?
\end{openquestion}

\section{Acknowledgments}

This research was supported by Riset Unggulan Perguruan Tinggi 2015, funded by Indonesian Ministry of Research, Technology, and Higher Education.

\end{document}